\documentclass[12pt, draft]{amsart}

\usepackage{amssymb}
\usepackage{amsthm}
\usepackage{amsmath}
\usepackage{amsfonts}
\usepackage[mathscr]{eucal}
\usepackage{enumerate}

\newtheorem{thm}{Theorem}[section]
\newtheorem{cor}[thm]{Corollary}

\theoremstyle{definition}

\theoremstyle{remark}

\numberwithin{equation}{section}

\begin{document}
\title[Hyperstability of a functional equation]
{Hyperstability of a functional equation}

\author{Eszter Gselmann}
\address{
Institute of Mathematics\\
University of Debrecen\\
P. O. Box: 12.\\
Debrecen\\
Hungary\\
H--4010}
\email{gselmann@math.klte.hu}

\begin{abstract}
The aim of this paper is to prove that the parametric fundamental equation of information
is hyperstable on its open as well as on its closed domain, assuming that the parameter
is negative. As a corollary of the main result, it is also proved that the system
of equation that defines the alpha-recursive information measures is stable.
\end{abstract}

\thanks{This research has been supported by the Hungarian Scientific Research Fund
(OTKA) Grants NK 68048 and K 62316.}
\maketitle

\section{Introduction}

The stability theory of functional equations mainly deals with the following question:
Is it true that the solution of a given equation differing slightly from a given one,
must necessarily be close to the solution of the equation in question?
In case of a positive answer we say that the equation in question is
\emph{stable} (see e.g. Forti \cite{For95}).
The investigation of the multiplicative Cauchy equation highlighted a new phenomenon,
which is nowadays called \emph{superstability} (see e.g. Moszner \cite{Mos04}).
In this case the so-called stability inequality implies that the function in question
is either bounded or it is the solution of the functional equation itself.
Furthermore, it can happen that there is no such an alternative, that is,
all the solutions of the stability inequality are exactly the solutions of the functional equation.
In this case we speak about \emph{hyperstability} (see e.g. Maksa--P\'{a}les \cite{MP01}).

In this paper we shall prove that the parametric fundamental equation of
information, that is,
\begin{equation}\label{Eq1.1}
f(x)+(1-x)^{\alpha}f\left(\frac{y}{1-x}\right)=
f(y)+(1-y)^{\alpha}f\left(\frac{x}{1-y}\right)
\end{equation}
is hyperstable, provided that the parameter $\alpha$ is negative.
Equation \eqref{Eq1.1} arises in the characterization problem
of the 3-semi-symmetric, $\alpha$-recursive information measures
(see Acz\'{e}l--Dar\'{o}czy \cite{AD75} and Ebanks--Sahoo--Sander \cite{ESS98}).

Throughout this paper we will use the following notations.
Let $\mathbb{R}$ denote the set of the real numbers and for each
$n=2, 3, \ldots$ we define the sets $D_{n}$ and $\overline{D}_{n}$, respectively, by
\[
D_{n}=\left\{(p_{1}, \ldots p_{n})\in\mathbb{R}^{n}\vert p_{1}, \ldots, p_{n}, \sum^{n}_{i=1}p_{i}\in ]0, 1[\right\}
\]

and
\[
\overline{D}_{n}=\left\{(p_{1}, \ldots p_{n})\in\mathbb{R}^{n}
\vert p_{1}, \ldots, p_{n}\in [0, 1[, \sum^{n}_{i=1}p_{i}\leq 1\right\}.
\]

In the investigations \eqref{Eq1.1} is supposed to hold on
$D_{2}$ or on $\overline{D}_{2}$.
The general solution of equation \eqref{Eq1.1} is contained in the
following theorem (see also Maksa \cite{Mak82} and Ebanks--Sahoo--Sander \cite{ESS98}).

\begin{thm}\label{Thm1.1}
Let $\alpha\in\mathbb{R}$, $\alpha\neq 1$ and $f:]0,1[\rightarrow\mathbb{R}$ be
a function. Then functional equation \eqref{Eq1.1} holds for all $(x, y)\in D_{2}$
if and only if there exist $c, d\in\mathbb{R}$ such that
\[
f(x)=cx^{\alpha}+d(1-x)^{\alpha}-d
\]
holds for all $x\in ]0, 1[$.
\end{thm}

In \cite{Mak08} Maksa proved that \eqref{Eq1.1} is superstable on
$\overline{D}_{2}$ if $1\neq \alpha> 0$, but this method is
inappropriate on the set $D_{2}$ as well as if the parameter $\alpha$ is negative.
It was conjectured that equation \eqref{Eq1.1} is hyperstable but finally
stability could be proved (see \cite{GM08}).
However, this method was appropriate to prove superstability for the case
$1\neq \alpha>0$.
This unified method can be found in \cite{Gse08}.
Recently it turned out that with a different method we can prove hyperstability for
equation \eqref{Eq1.1} on $D_{2}$ as well as on $\overline{D}_{2}$, if $\alpha<0$.

The next section will be devoted to this result.

All in all, equation \eqref{Eq1.1} is superstable if $1\neq \alpha>0$, it is
stable, provided that $\alpha=0$ and it is hyperstable assuming that
$\alpha<0$ on the set $D_{2}$.
As to the case $\alpha=1$, it was raised by Sz\'{e}kelyhidi during the 28th
International Symposium on Functional Equations (see \cite{Szek91}), and it is still open.

\section{Main result}

Our main result is contained in the following theorem.

\begin{thm}\label{Thm2.1}
Let $\alpha, \varepsilon\in\mathbb{R}$, $\alpha<0$, $\varepsilon\geq 0$
and $f:]0, 1[\rightarrow\mathbb{R}$ be a function.
Assume that
\begin{equation}\label{Eq2.1}
\left|f(x)+(1-x)^{\alpha}f\left(\frac{y}{1-x}\right)-
f(y)-(1-y)^{\alpha}f\left(\frac{x}{1-y}\right)\right|\leq \varepsilon
\end{equation}
holds for all $(x, y)\in D_{2}$.
Then, and only then, there exist $c, d\in\mathbb{R}$ such that
\begin{equation}\label{Eq2.2}
f(x)=cx^{\alpha}+d(1-x)^{\alpha}-d
\end{equation}
for all $x\in ]0, 1[$.
\end{thm}

\begin{proof}
Due to Theorem \ref{Thm1.1}., in case the function $f$ is given by formula
(\ref{Eq2.2}), then
\[
f(x)+(1-x)^{\alpha}f\left(\frac{y}{1-x}\right)=
f(y)+(1-y)^{\alpha}f\left(\frac{x}{1-y}\right)
\]
holds for all $(x, y)\in D_{2}$.
Thus inequality (\ref{Eq2.1}) is also satisfied with arbitrary $\varepsilon\geq 0$.
Therefore it is enough to prove the converse direction.

Define the function $G:D_{2}\rightarrow\mathbb{R}$ by
\begin{equation}\label{Eq2.3}
G(x, y)=f(x)+(1-x)^{\alpha}f\left(\frac{y}{1-x}\right)-f(x+y).
\quad \left((x, y)\in D_{2}\right)
\end{equation}
Then inequality (\ref{Eq2.1}) immediately implies that
\begin{equation}\label{Eq2.4}
\left|G(x, y)-G(y, x)\right|\leq \varepsilon
\end{equation}
for all $(x, y)\in D_{2}$.

Let $(x, y, z)\in D_{3}$, then due to the definition of the function $G$,
\[
G(x+y, z)=f(x+y)+(1-(x+y))^{\alpha}f\left(\frac{z}{1-(x+y)}\right)-f(x+y+z),
\]

\[
G(x, y+z)=f(x)+(1-x)^{\alpha}f\left(\frac{y+z}{1-x}\right)-f(x+y+z)
\]
and

\begin{multline*}
(1-x)^{\alpha}G\left(\frac{y}{1-x}, \frac{z}{1-x}\right)\\
=(1-x)^{\alpha}\left[f(\frac{y}{1-x})+
\left(1-\frac{y}{1-x}\right)^{\alpha}f\left(\frac{\frac{z}{1-x}}{1-\frac{y}{1-x}}\right)-
f\left(\frac{y+z}{1-x}\right)\right],
\end{multline*}
therefore
\begin{equation}\label{Eq2.5}
G(x, y)+G(x+y, z)
=G(x, y+z)+(1-x)^{\alpha}G\left(\frac{y}{1-x}, \frac{z}{1-x}\right)
\end{equation}
holds on $D_{3}$, where we used the identity
\[
\frac{z}{1-(x+y)}=\frac{\frac{z}{1-x}}{1-\frac{y}{1-x}}
\]
also.

In what follows we will show that the function $G$ is $\alpha$--homogeneous.
Indeed, interchanging $x$ and $y$ in (\ref{Eq2.5}), we get
\begin{multline*}
G(y, x)+G(x+y, z)\\
=G(y, x+z)+(1-y)^{\alpha}G\left(\frac{x}{1-y}, \frac{z}{1-y}\right).
\quad \left((x, y, z)\in D_{3}\right)
\end{multline*}
Furthermore, equation (\ref{Eq2.5}) with the substitution
\[
(x, y, z)=(y, z, x)
\]
yields that
\[
G(y, z)+G(y+z, x)=G(y, x+z)+(1-y)^{\alpha}G\left(\frac{z}{1-y}, \frac{x}{1-y}\right)
\]
is fulfilled for all $(x, y, z)\in D_{3}$.

Thus
\begin{equation}\label{Eq2.6}
\begin{array}{l}
G(y, z)-(1-x)^{\alpha}G\left(\frac{y}{1-x}, \frac{z}{1-x}\right) \\
=\left\{G(x, y)+G(x+y, z)-G(x, y+z)-(1-x)^{\alpha}G\left(\frac{y}{1-x}, \frac{z}{1-x}\right)\right\} \\
-G(x, y)-G(x+y, z)+G(x, y+z) \\
+\left\{G(y, x)+G(x+y, z)-G(y, x+z)-(1-y)^{\alpha}G\left(\frac{x}{1-y}, \frac{z}{1-y}\right)\right\} \\
+\left\{G(y, z)+G(y+z, x)-G(y, x+z)-(1-y)^{\alpha}G\left(\frac{z}{1-y}, \frac{x}{1-y}\right)\right\} \\
-G(y+z, x)+G(y, x+z)+(1-y)^{\alpha}G\left(\frac{z}{1-y}, \frac{x}{1-y}\right) \\
=G(y, x)-G(x, y) +G(x, y+z)-G(y+z, x) \\
+(1-y)^{\alpha}\left(G\left(\frac{z}{1-y}, \frac{x}{1-y}\right)-G\left(\frac{x}{1-y}, \frac{z}{1-y}\right)\right)
\end{array}
\end{equation}
for all $(x, y, z)\in D_{3}$, since the expressions in the curly brackets are zeros.
Thus \eqref{Eq2.6}, \eqref{Eq2.4} and the triangle inequality imply that
\begin{equation}\label{Eq2.6*}
\left|G(y, z)-(1-x)^{\alpha}G\left(\frac{y}{1-x}, \frac{z}{1-x}\right)\right|
\leq \left(2+(1-y)^{\alpha}\right)\varepsilon
\end{equation}
is fulfilled for all $(x, y, z)\in D_{3}$.
Given any $t\in ]0, 1[$, $(u, v)\in D_{2}$, let
\[
x=1-t, \quad y=tu \quad \text{and} \quad z=tv.
\]
Then $x, y, z\in ]0, 1[$ and
\[
x+y+z=1-t(1-u-v)\in ]0, 1[,
\]
that is $(x, y, z)\in D_{3}$, and inequality (\ref{Eq2.6*}) implies that
\[
\left|G(tu, tv)-t^{\alpha}G(u, v)\right|\leq \left(2+(1-tu)^{\alpha}\right)\varepsilon,
\]
or, after rearranging,
\[
\left|\frac{G(tu, tv)}{t^{\alpha}}-G(u, v)\right|\leq \frac{\left(2+(1-tu)^{\alpha}\right)}{t^{\alpha}}\varepsilon
\]
holds for arbitrary $t\in ]0, 1[$ and $(u, v)\in D_{2}$.
Taking the limit $t\rightarrow 0+$ we obtain that
\[
\lim_{t\rightarrow 0+}\frac{G(tu, tv)}{t^{\alpha}}=G(u, v), \quad \left((u, v)\in D_{2}\right)
\]
since $\lim_{t\rightarrow 0+}(1-tu)^{\alpha}=1$ for all $u\in ]0, 1[$ and
$\lim_{t\rightarrow 0+}t^{-\alpha}=0$, since $\alpha<0$.
This implies that the function $G$ is $\alpha$--homogeneous on $D_{2}$.
Indeed, for arbitrary $s\in ]0, 1[$ and $(u, v)\in D_{2}$

\begin{equation}\label{Eq2.7}
\begin{array}{lcl}
G(su, sv)&=&\lim_{t\rightarrow 0+}\frac{G(t(su), t(sv))}{t^{\alpha}}\\
&=&s^{\alpha}\lim_{t\rightarrow 0+}\frac{G\left((ts)u, (ts)v\right)}{(ts)^{\alpha}}=s^{\alpha}G(u, v).
\end{array}
\end{equation}
At this point of the proof we will show that inequality \eqref{Eq2.4} and equation
\eqref{Eq2.7} together imply the symmetry of the function $G$.
Indeed, due to \eqref{Eq2.4}
\[
\left|G\left(tx, ty\right)-G\left(ty, tx\right)\right|\leq \varepsilon
\]
holds for all $(x, y)\in D_{2}$ and $t\in ]0, 1[$.
Using the $\alpha$-homogeneity of the function $G$, we obtain that
\[
\left|t^{\alpha}G\left(x, y\right)-t^{\alpha}G\left(y, x\right)\right|\leq\varepsilon,
\quad \left((x, y)\in D_{2}, t\in ]0, 1[\right)
\]
or, if we rearrange this,
\[
\left|G\left(x, y\right)-G\left(y, x\right)\right|\leq \frac{\varepsilon}{t^{\alpha}}
\]
holds for all $(x, y)\in D_{2}$ and $t\in ]0, 1[$. Taking the limit $t\rightarrow 0+$, we get
that
\[
G(x, y)=G(y, x)
\]
is fulfilled for all $(x, y)\in D_{2}$, since $\alpha<0$.
Therefore the function $G$ is symmetric. Due to definition \eqref{Eq2.3} this
implies that
\[
f(x)+(1-x)^{\alpha}f\left(\frac{y}{1-x}\right)=
f(y)+(1-y)^{\alpha}f\left(\frac{x}{1-y}\right),  \quad \left((x, y)\in D_{2}\right)
\]
i.e., the function $f$ satisfies the parametric fundamental equation of information on
$D_{2}$. Thus by Theorem \ref{Thm1.1}. there exist $c, d\in\mathbb{R}$
such that
\[
f(x)=cx^{\alpha}+d(1-x)^{\alpha}-d
\]
holds for all $x\in ]0, 1[$.
\end{proof}
\section{Corollaries}

The first corollary of this section says that equation \eqref{Eq1.1}
is hyperstable not only on $D_{2}$ but also on $\overline{D}_{2}$.

\begin{thm}
Let $\alpha, \varepsilon \in\mathbb{R}$ be fixed, $\alpha<0$, $\varepsilon\geq 0$.
Then the function $f:[0,1]\rightarrow\mathbb{R}$ satisfies the inequality
\eqref{Eq2.1} for all $(x, y)\in \overline{D}_{2}$ if,
and only if, there exist $c, d\in\mathbb{R}$
such that
\begin{equation}\label{Eq3.1}
f(x)=\left\{
\begin{array}{lcl}
0, & \mbox{if} & x=0 \\
cx^{\alpha}+d\left(1-x\right)^{\alpha}-d, & \mbox{if} & x\in ]0,1[ \\
c-d, & \mbox{if} & x=1.
\end{array}
\right.
\end{equation}
\end{thm}
\begin{proof}
Let $y=0$ in (\ref{Eq2.1}). Then we have that
\[
\left(\left(1-x\right)^{\alpha}+1\right)\left|f\left(0\right)\right|\leq \varepsilon
\quad \left(x\in ]0,1[\right)
\]
Since $\alpha<0$, this yields that $f\left(0\right)=0$.
On the other hand, by Theorem \ref{Thm2.1},
\[
f\left(x\right)=c x^{\alpha}+d\left(1-x\right)^{\alpha}-d
\quad \left(x\in ]0, 1[\right)
\]
with some $c, d\in\mathbb{R}$. Finally, let $x\in ]0, 1[$ and $y=1-x$ in
\eqref{Eq2.1}. Then, again by Theorem \ref{Thm2.1}., there exist $c, d\in \mathbb{R}$
such that
\[
\left|c-d-f\left(1\right)\right|
\left|x^{\alpha}-\left(1-x\right)^{\alpha}\right|\leq
\varepsilon.
\]
Since $\alpha<0$, $f\left(1\right)=c-d$ follows. \\
The converse is an easy computation and it turns out that
$f$ defined by \eqref{Eq3.1} is a solution of \eqref{Eq1.1} on $\overline{D}_{2}$.
\end{proof}

The second corollary concerns a system of equations.

For fixed $\alpha<0$ and $2\leq n\in\mathbb{N}$ define the set
\[
\Gamma^{\circ}_{n}=
\left\{(p_{1}, \ldots, p_{n})\vert p_{i}\in ]0, 1[, i=1, \ldots, n, \sum^{n}_{i=1}p_{i}=1\right\},
\]
and the function $H^{\alpha}_{n}:\Gamma^{\circ}_{n}\rightarrow\mathbb{R}$ by
\[
H^{\alpha}_{n}(p_{1}, \ldots, p_{n})=
\left(2^{1-\alpha}-1\right)^{-1}\cdot \left(\sum^{n}_{i=1}p^{\alpha}_{i}-1\right).
\]
The sequence $(H^{\alpha}_{n})$ is called \emph{the entropy of degree alpha},
and it was investigated among others by Dar\'{o}czy \cite{Dar70}, Havrda--Charv\'{a}t \cite{HC67}
and Tsallis \cite{Tsa88}, from several points of view.

About this topic the reader could consult the monograph of Acz\'{e}l and Dar\'{o}czy \cite{AD75}.
However, we will use only the following properties.

The sequence of functions $I_{n}:\Gamma^{\circ}_{n}\rightarrow\mathbb{R}$ ($n=2, 3, \ldots$)
is called \emph{$\alpha$-recursive}, if
\begin{multline*}
I_{n+1}(p_{1}, \ldots, p_{n+1})=\\
I_{n}(p_{1}+p_{2}, p_{3}, \ldots, p_{n+1})+
(p_{1}+p_{2})^{\alpha}I_{n}\left(\frac{p_{1}}{p_{1}+p_{2}}, \frac{p_{2}}{p_{1}+p_{2}}\right)
\end{multline*}
holds for all $n\geq 2$ and $(p_{1}, \ldots, p_{n+1})\in\Gamma^{\circ}_{n+1}$.

Furthermore, the sequence of functions
$I_{n}:\Gamma^{\circ}_{n}\rightarrow\mathbb{R}$ ($n=2, 3, \ldots$)
is called \emph{$3$-semi-symmetric}, if
\[
I_{3}(p_{1}, p_{2}, p_{3})=I_{3}(p_{1}, p_{3}, p_{2})
\]
holds on $\Gamma^{\circ}_{3}$.

It is easy to see that for all $a, b\in\mathbb{R}$ the sequence of functions
\[
(p_{1}, \ldots, p_{n})\mapsto aH^{\alpha}_{n}(p_{1}, \ldots, p_{n})+b(p_{1}^{\alpha}-1)
\quad \left((p_{1}, \ldots, p_{n})\in\Gamma^{\circ}_{n}\right)
\]
is $\alpha$-recursive and $3$-semi-symmetric. Therefore, the second
corollary can be considered as a stability theorem for a system of equations.

\begin{thm}\label{Thm3.2}
Let $n\geq 2$ be a fixed positive integer,
$\left(I_{n}\right)$ be the sequence of functions
$I_{n}:\Gamma^{\circ}_{n}\rightarrow\mathbb{R}$ and suppose that there exist a sequence
$\left(\varepsilon_{n}\right)$ of nonnegative real numbers and a real number
$\alpha<0$ such that
\begin{multline}\label{Eq3.2}
\left|I_{n}\left(p_{1}, \ldots, p_{n}\right)-\right. \\
\left. I_{n-1}\left(p_{1}+p_{2}, p_{3}, \ldots, p_{n}\right)-
\left(p_{1}+p_{2}\right)^{\alpha}I_{2}\left(\frac{p_{1}}{p_{1}+p_{2}},
\frac{p_{2}}{p_{1}+p_{2}}\right)\right|
\leq \varepsilon_{n-1}
\end{multline}
holds for all $n\geq 3$ and $\left(p_{1}, \ldots, p_{n}\right)\in\Gamma^{\circ}_{n}$,
and
\begin{equation}\label{Eq3.4}
\left|I_{3}\left(p_{1}, p_{2}, p_{3}\right)-I_{3}\left(p_{1}, p_{3}, p_{2}\right)\right|\leq
\varepsilon,
\end{equation}
holds on $\Gamma^{\circ}_{3}$.
Then there exist $a, b\in\mathbb{R}$ such that
\begin{equation}\label{Eq3.5}
\left|
I_{n}\left(p_{1}, \ldots, p_{n}\right)-
\left(a H^{\alpha}_{n}\left(p_{1}, \ldots, p_{n}\right)+b\left(p^{\alpha}_{1}-1\right)\right)\right|\leq \sum^{n-1}_{k=2}\varepsilon_{k}
\end{equation}
for all $n\geq 2$ and $\left(p_{1}, \ldots, p_{n}\right)\in\Gamma^{\circ}_{n}$,
where the convention $\sum^{1}_{k=2}\varepsilon_{k}=0$ is adopted.
\end{thm}

\begin{proof}
As in Maksa \cite{Mak08}, it can be proved that, due to
(\ref{Eq3.2}) and (\ref{Eq3.4}), for the function $f$
defined by $f(x)=I_{2}\left(1-x, x\right)$, $x\in ]0,1[$ we get that
\[
\left|
f\left(x\right)+\left(1-x\right)^{\alpha}f\left(\frac{y}{1-x}\right)-
f\left(y\right)-\left(1-y\right)^{\alpha}f\left(\frac{x}{1-y}\right)\right|
\leq 2\varepsilon_{2}+\varepsilon_{1}
\]
for all $\left(x, y\right)\in D_{2}$, i.e., \eqref{Eq2.1}
holds with $\varepsilon=2\varepsilon_{2}+\varepsilon_{1}$.
Therefore, applying Theorem \ref{Thm2.1}., we obtain
\eqref{Eq2.2} with some $c, d\in\mathbb{R}$, i.e.,
\[
I_{2}\left(1-x, x\right)=cx^{\alpha}+ d \left(1-x\right)^{\alpha}-d,
\quad \left(x\in ]0,1[\right)
\]
i.e., \eqref{Eq3.5} holds for $n=2$ with $a=(2^{1-\alpha}-1)c$, $b=d-c$.

We continue the proof by induction on $n$.
Suppose that (\ref{Eq3.5}) holds and, for the sake of brevity, introduce the
notation
\[
J_{n}\left(p_{1}, \ldots, p_{n}\right)=
a H^{\alpha}_{n}\left(p_{1}, \ldots, p_{n}\right)+b\left(p^{\alpha}_{1}-1\right)
\]
for all $n\geq 2$, $\left(p_{1}, \ldots, p_{n}\right)\in\Gamma^{\circ}_{n}$.
It can easily be seen that (\ref{Eq3.2}) and (\ref{Eq3.4}) hold on
$\Gamma^{\circ}_{n}$ for $J_{n}$ instead of $I_{n}$ $(n\geq 3)$ with
$\varepsilon_{n}=0$ $(n\geq 2)$.
Thus for all $\left(p_{1}, \ldots, p_{n+1}\right)\in\Gamma^{\circ}_{n+1}$,
we get that

\begin{multline*}
I_{n+1}\left(p_{1}, \ldots, p_{n+1}\right)-J_{n+1}\left(p_{1}, \ldots, p_{n+1}\right) \\
= I_{n+1}\left(p_{1}, \ldots, p_{n+1}\right)-
J_{n}\left(p_{1}+p_{2}, p_{3}, \ldots, p_{n+1}\right)\\-
\left(p_{1}+p_{2}\right)^{\alpha}
J_{2}\left(\frac{p_{1}}{p_{1}+p_{2}}, \frac{p_{2}}{p_{1}+p_{2}}\right) \\
= I_{n+1}\left(p_{1}, \ldots, p_{n+1}\right)-
I_{n}\left(p_{1}+p_{2}, p_{3}, \ldots, p_{n+1}\right)\\-
\left(p_{1}+p_{2}\right)^{\alpha}
I_{2}\left(\frac{p_{1}}{p_{1}+p_{2}}, \frac{p_{2}}{p_{1}+p_{2}}\right) \\
+ I_{n}\left(p_{1}+p_{2}, p_{3}, \ldots, p_{n+1}\right)-
J_{n}\left(p_{1}+p_{2}, p_{3}, \ldots, p_{n+1}\right) \\
+ \left(p_{1}+p_{2}\right)^{\alpha}
I_{2}\left(\frac{p_{1}}{p_{1}+p_{2}}, \frac{p_{2}}{p_{1}+p_{2}}\right)\\-
\left(p_{1}+p_{2}\right)^{\alpha}
J_{2}\left(\frac{p_{1}}{p_{1}+p_{2}}, \frac{p_{2}}{p_{1}+p_{2}}\right).
\end{multline*}

Therefore (\ref{Eq3.5}) with $n=2$ and the induction hypothesis imply that
\[
\left|
I_{n+1}\left(p_{1}, \ldots, p_{n+1}\right)-
J_{n}\left(p_{1}, \ldots, p_{n+1}\right)\right|\leq
\varepsilon_{n}+\sum^{n-1}_{k=2}\varepsilon_{k}=
\sum^{n}_{k=2}\varepsilon_{k},
\]
that is, (\ref{Eq3.5}) holds for $n+1$ instead of $n$.
\end{proof}

\begin{cor}
Applying Theorem \ref{Thm3.2}. with the choice $\varepsilon_{n}=0$
for all $n\in\mathbb{N}$, we get the $\alpha$-recursive, $3$-semi-symmetric
information measures. Hence the previous theorem says that the system of
$\alpha$-recursive and $3$-semi-symmetric information measures is stable.
\end{cor}

\noindent
\textbf{Acknowledgement.}
The author is grateful to Professor Gyula Maksa for his helpful comments.


\begin{thebibliography}{99}


\bibitem{AD75}
J.~Acz\'{e}l, Z.~Dar\'{o}czy,
\emph{On measures of information and their characterization},
Academic Press, New York -- San Francisco -- London, 1975.

\bibitem{Dar70}
Z.~Dar\'{o}czy, \emph{Generalized information functions}, Information
and Control \textbf{16} (1970), 36--51.

\bibitem{ESS98}
B.~R.~Ebanks, P.~Sahoo, W.~Sander,
\emph{Characterizations of information measures},
World Scientific Publishing Co., Inc., River Edge, NJ, 1998.

\bibitem{For95}
G. L.~Forti, \emph{Hyers--Ulam stability of functional equations in several variables}, Aequationes Math.
\textbf{50} (1995), no. 1--2, 143--190.

\bibitem{Gse08}
E.~Gselmann,
\emph{Recent results on the stability of the parametric fundamental equation of information},
Submitted to Acta Math. Acad. Paedagog. Nyh\'{a}zi.

\bibitem{GM08}
E.~Gselmann, Gy.~Maksa,
\emph{Stability of the parmetric fundamental equation of information for
nonpositive parameters},
To appear in Aequationes Math.

\bibitem{HC67}
J.~Havrda and F.~Charv\'{a}t,
\emph{Quantification Method of Classification Processes,
Concept of Structural $\alpha$-Entropy}, Kybernetika \textbf{3} (1967), 30--35.

\bibitem{Mak82}
Gy.~Maksa,
\emph{Solution on the open triangle of the generalized fundamental
equation of information with four unknown functions},
Utilitas Math. \textbf{21} (1982), 267--282.

\bibitem{Mak08}
Gy.~Maksa,
\emph{The stability of the entropy of degree alpha},
J. Math. Anal. Appl. \textbf{346} (2008), 17--21.

\bibitem{MP01}
Gy.~Maksa, Zs.~P\'{a}les,
\emph{Hyperstability of a class of linear functional equations},
Acta Math. Acad. Paedagog. Nyh\'{a}zi.  \textbf{17} (2001), no. 2, 107-–112.

\bibitem{Mos04}
Z.~Moszner,
\emph{Sur les d\'{e}finitions diff\'{e}rentes de la stabilit\'{e} des \'{e}quations fonctionnelles},
Aequationes Math.  \textbf{68}  (2004),  no. 3, 260--274.

\bibitem{Szek91}
L.~Sz\'{e}kelyhidi,
\emph{38. Problem} (in Report of Meeting), Aequationes Math. \textbf{41} (1991), 302.

\bibitem{Tsa88}
C.~Tsallis,
\emph{Possible Generalization of Boltzmann-Gibbs Statistics},
Journal of Statistical Physics \textbf{52(1-2)} (1988), 479--487.






\end{thebibliography}
\end{document}